\documentclass[12pt,reqno]{amsart}
\usepackage{graphicx}
\usepackage{amssymb,amsmath}
\usepackage{amsthm}
\usepackage{color,graphicx}
\usepackage{hyperref}
\usepackage{color}
\usepackage{mathabx}

\newcommand{\be}{\begin{equation}}

\newcommand{\ee}{\end{equation}}

\newcommand{\ve}{{\varepsilon}}

\newcommand{\kp}{{\kappa}}

\numberwithin{equation}{section}
\numberwithin{figure}{section}

\newtheorem{theorem}{Theorem}[section]
\newtheorem{proposition}[theorem]{Proposition}
\newtheorem{remark}[theorem]{Remark}
\newtheorem{lemma}[theorem]{Lemma}
\newtheorem{corollary}[theorem]{Corollary}
\newtheorem{definition}[theorem]{Definition}

\begin{document}
\vglue-1cm \hskip1cm
\title[Periodic Waves for the Klein-Gordon Equations]{Odd Periodic Waves for some Klein-Gordon Type Equations: Existence and Stability}

\begin{center}

\subjclass[2000]{76B25, 35Q51, 35Q70.}

\keywords{Klein-Gordon equations, Periodic waves, Orbital instability, Orbital stability.}

\maketitle

{\bf Guilherme de Loreno}

{Departamento de Matem\'atica - Universidade Estadual de Maring\'a\\
	Avenida Colombo, 5790, CEP 87020-900, Maring\'a, PR, Brazil.}\\
{ pg54136@uem.br}

{\bf F\'abio Natali}

{Departamento de Matem\'atica - Universidade Estadual de Maring\'a\\
Avenida Colombo, 5790, CEP 87020-900, Maring\'a, PR, Brazil.}\\
{ fmanatali@uem.br}

\vspace{3mm}

\end{center}

\begin{abstract}
In this paper, we establish the existence and stability properties of odd
periodic waves related to the Klein-Gordon type equations, which include the well known $\phi^4$ and $\phi^6$ models. Existence of periodic waves is determined by using a general planar theory of ODE. The spectral analysis for the corresponding linearized operator is established using the monotonicity of the period map combined with an improvement of the standard Floquet theory. Orbital stability in the odd sector of the energy space is proved using exclusively the monotonicity of the period map. The orbital instability of explicit solutions for the $\phi^4$ and $\phi^6$ models is presented using the abstract approach in \cite{grillakis1}.
\end{abstract}

\section{Introduction}
The main purpose of this paper is to present orbital stability properties  associated to the generalized Klein-Gordon equation
\begin{equation}\label{gKG}
\phi_{tt}-\phi_{xx}-\phi+\phi^{2k+1}=0,
\end{equation}
where $k\geq1$ is an integer and $\phi:\mathbb{R}\times[0,+\infty)\rightarrow \mathbb{R}$ is a periodic function at the spatial variable, that is, it satisfies $\phi(x+L,t)=\phi(x,t)$, where $(x,t)\in\mathbb{R}\times [0,+\infty)$ and $L>0$ indicates the minimal period of $\phi$. When $k=1$, we have the well known $\phi^4$ equation which plays an important role in nuclear and particle physics. For the case $k=2$, we obtain the $\phi^6$ equation found in the energy transport along the hydrogen-bonded chains. Both models have been studied by several researches along the last years.

Equation $(\ref{gKG})$ has an abstract Hamiltonian form,
\begin{equation}\label{hamilt31}
\Phi_t=JE'(\Phi),\end{equation} where $J$ is given by

\begin{equation}\label{J}
J=\left(\begin{array}{lll}
\ \ 0\ \ \ \  1\\\\
-1\ \ \ \ 0\end{array}\right),
\end{equation}
and $\Phi=(\phi,\psi)=(\phi,\phi_t)$. $E'$ in $(\ref{hamilt31})$ indicates the Fr\'echet derivative of the conserved quantity $E$ defined by

\begin{equation}\label{E}
E(\phi,\psi)
=\displaystyle\frac{1}{2}\int_{0}^{L}\left(\phi_x^2+\psi^2-\phi^2+\frac{\phi^{2k+2}}{k+1}\right)dx.
\end{equation}
Moreover, $(\ref{gKG})$ has another conserved quantity given by,
\begin{equation}\label{F}
F(\phi,\psi)=\displaystyle\int_{0}^{L}\phi_x\psi dx.
\end{equation}

\indent Equation $(\ref{gKG})$ admits kink and periodic traveling wave solutions of the form $\phi(x,t)=h(x-ct)$, where $c\in (-1,1)$ is the wave speed. Indeed, substituting it into the equation $(\ref{gKG})$, we obtain for $\omega=1-c^2>0$, the following ODE given by

\begin{equation}\label{ode1}
-\omega h''-h+h^{2k+1}=0.
\end{equation}
Defining $G(\phi,\psi)=E(\phi,\psi)-cF(\phi,\psi)$, one has by $(\ref{ode1})$ that $G'(h,ch')=0$. In other words, $(h,ch')$ is a critical point of the Lyapunov functional $G$.\\
\indent For the case $k=1$, we have the well known kink solution with hyperbolic tangent profile given by
\begin{equation}\label{tanh1}
h(x)=\tanh\left(\frac{x}{\sqrt{2\omega}}\right).
\end{equation}
\indent Henry, Perez and Werszinski \cite{henry} determined the 
orbital stability of $(\ref{tanh1})$ with respect to small perturbations in the energy space $X:=H_{per}^1([0,L])\times L_{per}^2([0,L])$. Kowalczyk, Martel and Mu\~noz in \cite{martel} proved the asymptotic stability of the kink for odd perturbations in the energy space. The authors based their proof on Virial-type estimates and it has been inspired in some results concerning the asymptotic stability of solitons for the generalized Korteweg-de Vries equations (see \cite{martel-merle2} and \cite{martel-merle1}).\\
\indent In periodic context, we can find an explicit solution depending on the Jacobi elliptic function of snoidal type as
\begin{equation}\label{SNsol1}
h(x)=\frac{\sqrt{2}\kp}{\sqrt{\kp^2+1}}{\rm sn}\left(\frac{4K(\kp)x}{L},\kp\right),
\end{equation} 
where $\kp\in(0,1)$ is called modulus of the elliptic function and $K(\kp)$ is the complete  elliptic integral of first kind. A nice result regarding periodic waves has been determined by Palacios \cite{palacios}, where the author used the abstract theory in \cite{grillakis1} adapted to the periodic context. However, the author performed the stability result by considering the Ginzburg-Landau energy given by
\begin{equation}\label{GLE}
\widetilde{E}(\phi,\psi)
=\displaystyle\frac{1}{2}\int_{0}^{L}\left(\phi_x^2+\psi^2+\left(1-\frac{\phi^{2}}{2}\right)^2\right)dx.
\end{equation}
As far as we can see, the stability analysis in the periodic context is the same if one considers our energy functional given by $(\ref{E})$ for the case $k=1$.\\ 
\indent Concerning the case $k=2$, kink solutions with hyperbolic tangent profile is given by
\begin{equation}\label{phitahn}
h(x)=\sqrt{\frac{2}{3}}\frac{{\rm tanh}\left(\frac{x}{\sqrt{\omega}}\right)}{\sqrt{1-\frac{1}{3}{\rm tanh}^2\left(\frac{x}{\sqrt{\omega}}\right)}}.
\end{equation}
Unfortunately, we do not know references concerning the orbital/asymptotic  stability for this model, but we believe that arguments in \cite{henry} can be repeated in order to get the orbital stability in the energy space $X$.\\
\indent Periodic waves with snoidal profile associated to the model $(\ref{ode1})$ are given by
\begin{equation}\label{SNSNso2}
h(x)=\frac{a{\rm sn}\left(\frac{4K(\kappa)x}{L},\kappa\right)}{\sqrt{1-b{\rm sn}^2\left(\frac{4K(\kappa)x}{L},\kappa\right)}},
\end{equation}
where $a$ and $b$ are complicated functions depending only on  the modulus $\kp\in (0,1)$ (see the definition of $a$ and $b$ in $(\ref{eqaSN})$ and $(\ref{bSNSN})$, respectively). The value of $\omega$ depends on $\kp$ and $L$ and it is expressed by
\begin{equation}\label{omegaphi6}
\omega=\frac{L^2}{16K(\kp)^2\sqrt{\kp^4-\kp^2+1}}.
\end{equation}
\indent  According with \cite{grillakis1} and since $J$ in $(\ref{J})$ is invertible with bounded inverse, the first requirement to get the orbital stability/instability of periodic waves for $(\ref{gKG})$, is to construct a smooth curve $\omega\in \mathcal{I}\subset(0,+\infty)\mapsto h_{\omega}\in H_{per}^2([0,L])$ with fixed period solving $(\ref{ode1})$. The second claim is to prove that
\begin{equation}
\displaystyle G''(h,ch')\equiv\mathcal{L}_{KG}=\left(
\begin{array}{lll}
-\partial_x^2-1+(2k+1)h^{2k}& &c\partial_x\\\\
\ \ \ \ -c\partial_x & & 1
\end{array}\right)
\label{matrixop313}\end{equation}
has only one negative eigenvalue which is simple and zero is a simple eigenvalue associated to the eigenfunction $(h',ch'')$. To do so, it makes necessary to combine the min-max characterization of eigenvalues with the study of the first non-positive eigenvalues associated to the single linear operator given by
\begin{equation}\label{linoperat}
\mathcal{L}=-\omega\partial_x^2-1+(2k+1)h^{2k}.
\end{equation}
 \indent Defining $d(c)=E(h,ch')-cF(h,ch')=G(h,ch')$, since $\omega=1-c^2$ and $\omega\mapsto h$ is a smooth curve of periodic waves with fixed period, we have from $G'(h,ch')=0$ that $d'(c)=-cF(h,ch')$. In this setting, if $d''(c)<0$, we have the orbital instability, and if $d''(c)>0$, the orbital stability. Here, $d''(c)$ is given by
   \begin{equation}\label{dsegunda}d''(c)=-\int_0^{L}(h'(x))^2dx+2(1-\omega)\frac{d}{d\omega}\int_0^{L}(h'(x))^2dx.\end{equation} 
\indent In \cite{palacios}, the author uses all requirements above to establish the orbital instability of the periodic wave $(\ref{SNsol1})$ in the energy space $X$. The crucial point to get the second requirement was the exact determination of the instability intervals associated to the Lam\'e equation
\begin{equation}\label{lame}
\left\{
\begin{array}{lll}
\displaystyle\frac{d^2}{dx^2}\Psi+[\rho-6\kp^2sn^2(x;\kp)]\Psi=0\\
\Psi(0)=\Psi(4K(\kp)),\;\;\Psi'(0)=\Psi'(4K(\kp)).
\end{array}
\right.
\end{equation}
 According with our best knowledge, for $k\neq1$, the linearized operator associated to the solution in $(\ref{SNSNso2})$ can not be reduced in a Lam\'e equation $(\ref{lame})$ and then, the arguments in \cite{palacios} can not be used to prove the orbital stability/instability related to the $\phi^6$ model.\\
 \indent To overcome this difficulty, we use a general planar analysis of ODE to deduce the existence of periodic solutions for the general equation $(\ref{gKG})$. In addition, let us define the period map $L=L(B)$ in terms of $B\in (0,B_{\omega})$ as 
\begin{equation}L:(0,B_{\omega})\mapsto \int_{\Gamma_B}\frac{dh}{\xi},
\label{permap}\end{equation}
where $B$ is the energy level of $\mathcal{E}(h,\xi)=\frac{\xi^2}{2}+\frac{h^2}{2\omega}-\frac{h^{2k+2}}{(2k+2)\omega}=B$, $\xi=h'$, $B_{\omega}=\frac{k}{2\omega(k+1)}$, and $\Gamma_B$ stands the orbit in the phase portrait corresponding to the periodic solution $h$ of $(\ref{ode1})$. Our analysis establishes that $L_B>0$ and this fact can be used to determine the required spectral properties associated to the linear operator $\mathcal{L}$ in $\ref{linoperat}$  for all $k\geq1$ integer. This last fact can be determined employing an improvement of the Floquet theory as in \cite{natali1}, \cite{neves} and \cite{neves1}. \\
\indent We present our main results. Without using explicit solutions, we  establish for the case $c=0$, the stability of periodic waves for the general equation $(\ref{gKG})$ in the odd sector $X_{odd}=H_{odd}^1([0,L])\times L_{odd}^2([0,L])$. For the case $k=1$, our paper improves the arguments in \cite{palacios} since it has been used the explicit solution in $(\ref{SNsol1})$ and the Lam\'e equation in $(\ref{lame})$ to obtain the stability in the odd sector. We use the monotonicity of the period map $L$ in $(\ref{permap})$ to obtain the following basic property 
\begin{equation}\label{quadraticQKG1}\langle\mathcal{L}_{KG}(u,v),(u,v)\rangle_X\geq \gamma||(u,v)||_{X}^2,\end{equation}
where $\gamma>0$. Inequality in $(\ref{quadraticQKG1})$ is the cornerstone to establish that $(h,0)$ is minimum point of the energy $E$ in $(\ref{E})$ for fixed values of the constraint $F$ in $(\ref{F})$. Thus, the abstract theory in \cite{grillakis1} can be applied to obtain the stability in $X_{odd}$.\\
\indent Concerning the $\phi^4$ and $\phi^6$ models with explicit solutions given by $(\ref{SNsol1})$ and $(\ref{SNSNso2})$, respectively. We see that our planar analysis of ODE gives us that $L_B>0$ for any $k\geq1$. Thus, according with \cite{neves}, \cite{neves1} and the min-max characterization of eigenvalues, we obtain that $\mathcal{L}_{KG}$ has only one negative eigenvalue which is simple and zero is a simple eigenvalue whose associated eigenfunction is $(h',ch'')$. The orbital instability is then proved by establishing that $d''(c)<0$. For the $\phi^4$, we calculate $d''(c)$ using the expression in $(\ref{dsegunda})$ and the explicit solution in $(\ref{SNsol1})$. Regarding the $\phi^6$ model, we use the explicit solution in $(\ref{SNSNso2})$ combined with simplified formula for $d''(c)$ to avoid heavy expressions containing the complete elliptic integral of third kind (see \cite{byrd}).\\
\indent Our paper is organized as follows. Section 2 is devoted to present the existence of periodic waves for the general equation $(\ref{gKG})$. Spectral analysis for the linearized operator $\mathcal{L}_{KG}$ is established in Section 3. The stability in the odd sector $X_{odd}$ for the general solution of $(\ref{ode1})$ as well as the orbital instability in $X$ for the $\phi^4$ and $\phi^6$ models are presented in Section 4.

\section{Existence of odd periodic waves and the monotonicity of the period-map.}
Our purpose in this section is to present some facts concerning the existence of
periodic solutions for the nonlinear ODE written as
\be
 - \omega h'' -h + h^{2k+1} = 0, \label{ode}
\ee
where $\omega>0$ and $k\geq 1$ is an integer. 

 It is well known that equation (\ref{ode}) is
conservative, and thus its solutions are contained on the level
curves of the energy
\begin{equation}\label{energyODE}
\mathcal{E}(h, \xi) :=  \frac{\xi^2}{2} +\frac{h^2}{2\omega}-\frac{h^{2k+2}}{(2k+2)\omega},
\end{equation}
where $\xi=h'$.\\
\indent According to the classical ODE theory (see \cite{chicone}, \cite{jack} and \cite{natali1} for further details), $h$ is a periodic solution of the equation $(\ref{ode})$ if, and only if, $(h,h')$ is a periodic orbit of the planar differential system
\begin{equation}\label{planarODE}
\left\{\begin{array}{lllll}
h'=\xi,\\\\
\xi=\displaystyle-\frac{h}{\omega}+\frac{h^{2k+1}}{\omega}.
\end{array}\right.
\end{equation}
\indent The periodic orbits associated to the equation $(\ref{planarODE})$ lies inside of convenient energy levels of the energy $\mathcal{E}$. This means that the pair $(h,h')$ satisfies the equation $\mathcal{E}(h,h')=B$ for all $B\in(0,B_{\omega})$, where $B_{\omega}=\frac{k}{2\omega(k+1)}$. Moreover, the periodic orbits turn round at the center critical point of the system $(\ref{planarODE})$. In our case, we see that $(\ref{planarODE})$ has three critical points, being one center point at $(h,h')=(0,0)$, and two saddle points at $(h,h')=(\pm 1,0)$. According to the standard ODE theory, the periodic orbits emanate from the center point to the separatrix curve which is represented by a smooth solution $\widetilde{h}:\mathbb{R}\rightarrow\mathbb{R}$ of $(\ref{ode})$ satisfying $\lim_{x\rightarrow \pm \infty}\widetilde{h}(x)=\pm 1$ and $\widetilde{h}'(x)>0$ for all $x\in\mathbb{R}$.\\
\indent From the analysis above, the periodic orbits of the planar system $(\ref{planarODE})$ correspond to odd periodic solutions $h$ of $(\ref{ode})$. The period $L=L(\omega,B)$ of the solution $h$ can be expressed (formally) by

\be
L=\displaystyle 2\int_{b_1}^{b_2}\frac{dh}{\sqrt{\frac{h^{2k+2}}{\omega(k+1)}-\frac{h^2}{\omega}+2B}},
\label{persol}\ee
where $b_1=\min_{x\in [0,L]}h(x)$ and $b_2=\max_{x\in [0,L]}h(x)$.\\
\indent On the other hand, the energy levels of the first integral $\mathcal{E}$ in $(\ref{energyODE})$ parametrize the set of periodic orbits $\{\Gamma_B\}_{B\in(0,B_{\omega})}$ which emanate from the center point to the separatrix curve. Thus, we can conclude that the set of smooth periodic solutions of $(\ref{ode})$ can be expressed by a smooth family $h=h_{\omega,B}$ which is parametrized by $\omega>0$ and $B\in(0,B_{\omega})$. Thus, we obtain that the orbit $\Gamma_B$ is equal to the corresponding periodic solution $h$. In addition, since the period of the orbit is given by the smooth map
\be
\widetilde{L}=\int_{\Gamma_B}\frac{dh}{\xi},
\label{perorbit}
\ee
we see from $(\ref{persol})$ and $(\ref{perorbit})$ that $L=\widetilde{L}$. \\
\indent Moreover, if $B\rightarrow 0$ we have $L\rightarrow  \alpha(\omega)$, and if $B\rightarrow B_{\omega}$, one has $L\rightarrow +\infty$, where $\alpha(\omega)>0$ is the corresponding period of the equilibrium center point. Next result gives us the monotonicity of the period-map $L$ defined in $(\ref{perorbit})$.

\begin{lemma} For every $\omega>0$ and $k\geq1$ integer, the funtion
	$$L:(0,B_{\omega})\rightarrow \mathbb{R}_{+},\ \ \ \ B\mapsto \int_{\Gamma_B}\frac{dh}{\psi}$$
	 satisfies $L_{B}>0$.
	
\label{lemaper1}\end{lemma}
\begin{proof}
	Consider $f(h)=h-h^{2k+1}$ and $F$ satisfying $F'=f$. According to the  \cite[Section 4]{bonorino}, it suffices to prove that
	$$I(h)=\frac{F'(h)^2-2F(h)F''(h)}{F'(h)^3}$$
	is a strictly increasing function over the interval $(-1,1)$. Doing the computations, we obtain that
	
	$$I(h)=-kh^{2k-1}\frac{1+2k-h^{2k}}{(1+k)(h^{2k}-1)^3},$$
	that is, $I$ is smooth in the interval $(-1,1)$. In addition, an exhaustive calculation gives us
	
	$$\begin{array}{lll}I'(h)&=& -\displaystyle \frac{kh^{2k-2}}{(k+1)(h^{2k}-1)^4}\left[(-2k-8k^2-2)h^{2k}+(2k+1)h^{4k}+1-4k^2\right]\\\\
	&\geq&-\displaystyle\frac{kh^{2k-2}}{(k+1)(h^{2k}-1)^4}(-4k^2-2)h^{2k}>0.
	\end{array}$$
\end{proof}

\section{Spectral Properties.}
\subsection{Floquet Theory Framework}
We need to recall some basic facts concerning Floquet's theory (see \cite{est} and \cite{magnus}). Let $Q$ be a smooth $L$-periodic function.  Consider $\mathcal{P}$ the Hill operator defined in $L_{per}^2([0,L])$, with domain $D(\mathcal{P})=H_{per}^2([0,L])$, given by
$$
\mathcal{P}=-\partial_x^2+Q(x).
$$
The spectrum of $\mathcal{P}$  is formed by an unbounded sequence of
real eigenvalues
\[
\lambda_0 < \lambda_1 \leq \lambda_2 \leq \lambda_3 \leq \lambda_4\leq
\cdots\; \leq \lambda_{2n-1} \leq \lambda_{2n}\; \cdots,
\]
where equality means that $\lambda_{2n-1} = \lambda_{2n}$  is a
double eigenvalue. Moreover, according with the Oscillation Theorem, the spectrum  is characterized by the number of zeros
of the eigenfunctions as: if $p$ is an eigenfunction associated to either $\lambda_{2n-1}$ or $\lambda_{2n}$, then $p$  has exactly
$2n$ zeros in the half-open
interval $[0, L)$. In particular, the even eigenfunction associated to the first eigenvalue $\lambda_0$ has no zeros in $[0, L]$.

Let $p(x)$ be a nontrivial $L$-periodic solution of the equation
\begin{equation}\label{zeqL}
\mathcal{P}f\equiv-f''+Q(x)f=0.
\end{equation}
Consider $y(x)$ the another solution of \eqref{zeqL} linearly independent of $p(x)$.  There exists a constant $\theta$ (depending on $y$ and $p$) such that (see \cite[page 5]{magnus})
\be\label{theta1}
y(x+L)=y(x)+\theta p(x).
\ee
Consequently, $\theta=0$ is a necessary and sufficient condition to all solutions of \eqref{zeqL} be $L$-periodic. This criterion is very useful to establish if the kernel of $\mathcal{P}$ is one-dimensional.

Next, for a fixed $\omega>0$ and $k\geq1$ integer, let $h=h_{\omega,B}$ be any periodic solution of \eqref{ode}. Consider $\mathcal{L}_{\omega,B}$ the linearized operator arising from the linearization of \eqref{ode} at $h=h_{\omega,B}$, that is,
\begin{equation}\label{operator}
\mathcal{L}_{\omega,B}(y):=-\omega y''-y+(2k+1)h^{2k}y.
\end{equation}
We see that $\mathcal{P}$ and $\mathcal{L}_{\omega,B}$ can be related as $\mathcal{P}=\frac{1}{\omega}\mathcal{L}_{\omega,B}$ and $Q(x)=-\frac{1}{\omega}+\frac{2k+1}{\omega}h^{2k}$. By taking the derivative with respect to $x$ is \eqref{ode}, we see that $h'$ belongs to the kernel of the operator  $\mathcal{L}_{\omega,B}$. In addition, from our construction, $h'$ has exactly
two zeros in the half-open interval $[0,L )$, which implies that zero is the second or the third eigenvalue of  $\mathcal{L}_{\omega,B}$. The next result gives that it is possible to decide the exact position of the zero eigenvalue by knowing the precise sign of $\theta$ in $(\ref{theta1})$.

\begin{lemma}\label{specprop}
Assume that $\theta$ in $(\ref{theta1})$ satisfies $\theta<0$. The operator $\mathcal{L}_{\omega,B}$, defined in $L_{per}^2([0,L])$, with domain $H_{per}^2([0,L])$, has exactly one negative eigenvalue, a simple eigenvalue at zero and the rest of the spectrum is positive and bounded away from zero.
\end{lemma}
\begin{proof}
See Theorem 3.1 in \cite{neves}.
\end{proof}

Next, let $\bar{y}$ be the unique solution of the initial-value problem
\be \left\{
\begin{array}{l}
  -\omega\bar{y}'' -\bar{y}+(2k+1)h^{2k}\bar{y} = 0 ,\\
 \bar{y}(0) =0, \\
\bar{y}'(0)=  \frac{1}{ h'(0)},
 \end{array} \right.
\label{y}
\ee
Since $h'$ is an $L$-periodic solution of the equation in \eqref{y} and the Wronskian of $\bar{y}$ and $h'$ is 1, there is a constant $\theta=\theta_{\bar{y}}$ such that
\begin{equation}\label{y1}
\bar{y}(x+L)=\bar{y}(x)+\theta h'(x).
\end{equation}
By taking the derivative in this last expression and evaluating at $x=0$, we obtain
\be \label{theta}
\theta= \frac{
\bar{y}(L)}{h'(0)}.
\ee
Now we can give the equality between $L_B$ and $\theta$.

\begin{lemma} \label{teo1}
We have, $L_B=-\theta$, where $\theta$ is the constant in \eqref{theta}.
\end{lemma}
\begin{proof}
Consider $\bar{y}$ and $h'$ as above. Since $h$ is odd and periodic one has $h(0)=h(L)=0$. Thus, the smoothness of $h$ in terms of the parameter $B$ enables us to take the derivative of $h(L)=0$ with respect to $B$ to obtain 
\begin{equation}\label{eq12343}
h'(L)L_B+\frac{\partial h(L)}{\partial B}=0.\end{equation}
\indent Next, we turn back to equation $(\ref{ode})$ and multiply it by $h'$ to deduce, after integration, the quadrature form
\begin{equation}\label{eq12344}
\frac{\omega h'^2(x)}{2}+\frac{h(x)^2}{2}-\frac{h(x)^{2k+2}}{2k+2}=\omega B,\ \ \ \ \mbox{for all}\ x\in [0,L].
\end{equation}
Deriving equation $(\ref{eq12344})$ with respect to $B$ and taking $x=0$ in the final result, we obtain from $(\ref{ode})$ that 
$\frac{\partial h'(0)}{\partial B}=\frac{1}{ h'(0)}$. In addition, since $h$ is odd one has that $\frac{\partial h}{\partial B}$ is also odd and thus $\frac{\partial h(0)}{\partial B}=0$. The existence and uniqueness theorem for ordinary differential equations applied to the problem $(\ref{y})$ enables us to deduce that $\bar{y}=\frac{\partial h}{\partial B}$. Therefore, we can combine $(\ref{theta})$ with $(\ref{eq12343})$ to obtain that $L_B=-\theta$. 
\end{proof}

Next, under our assumptions, $\theta$ does not change sign when $\omega$ and $B$ vary. Given any periodic solution $h=h_{\omega, B}$, we denote by $\theta_{\omega,B}=\theta_{\bar{y}_{\omega,B}}$ its corresponding constant as in \eqref{y1}.

\begin{definition}\label{defi12}
Let $Q$ be an smooth $L$-periodic function. Let $\mathcal{P}$ be the Hill operator defined in $L_{per}^2([0,L])$ with domain $D(\mathcal{P})=H_{per}^2([0,L])$ given by
$$
\mathcal{P}=-\partial_x^2+Q(x).
$$
 The inertial index of $\mathcal{P}$, denoted by  $in(\mathcal{P})$ is the pair $(n, z)$, where $n$ denotes the dimension of the negative
subspace of $\mathcal{P}$ and $z$ denotes the dimension of  $\ker(\mathcal{P})$.
\end{definition}

\begin{definition}\label{defi1234}
Assume that $Q(x)=Q_{\omega}(x)$ is periodic and depends on the parameter $\omega>0$.
 The family of linear operators
 $\mathcal{P}_{\omega}:=-\partial_x^2+Q_{\omega}(x)$, is said to be isoinertial if
 $in(\mathcal{P}_{\omega})$ is constant for any $\omega>0$.
\end{definition}

\begin{proposition} \label{teoisoL}
Let $h_{\omega,B}$ be the family of solution determined in previous section. Then the
family of linear operators $\mathcal{L}_{\omega,B}=-\omega\partial_x^2-1+(2k+1)h^{2k}$, is isoinertial. In particular, if $\theta_{\omega_0,B_0}<0$ for some $(\omega_0,B_0)$ with $B_0\in(0,B_{\omega_0})$, then $\theta_{\omega,B}<0$ for any $(\omega,B)\in(0,+\infty)\times(0,B_{\omega})$.
\end{proposition}
\begin{proof}
See Theorem 3.1 in \cite{natali1}.
\end{proof}

\begin{remark}\label{obs123}
Theorem \ref{teoisoL} establishes that in order to calculate the inertial index of $\mathcal{L}_{\omega,B}$ it suffices to calculate it for any fixed pair $(\omega_0,B_0)$.
\end{remark}

\begin{remark}\label{obs1.1}
	For a fixed $\omega>0$, we have by Lemma $\ref{lemaper1}$ that $\theta_{\omega,B}=-L_B<0$ for all $B\in(0,B_{\omega})$. Therefore, $in(\mathcal{L}_{\omega,B})=(1,1)$ which means that zero is a simple eigenvalue corresponding to the eigenfunction $h_{\omega,B}'$ and $n(\mathcal{L}_{\omega,B})=1$, where $n(\mathcal{A})$ stands the number of negative eigenvalues of a certain linear operator $\mathcal{A}$. From Floquet theory, we see that the eigenfunction associated to the negative eigenvalue is simple, even and it can be assumed positive over $[0,L]$.
	\end{remark}

\subsection{Spectral Analysis.} Let $h=h_{\omega,B}$ be the periodic solution of $(\ref{ode})$ where $\omega=1-c^2>0$. In
this subsection, we are going to analyze the quantity and multiplicity of the non-positive eigenvalues related to the matrix operator given by

\begin{equation}
\displaystyle\mathcal{L}_{KG}=\left(
\begin{array}{lll}
-\partial_x^2-1+(2k+1)h^{2k}& &c\partial_x\\\\
\ \ \ \ -c\partial_x & & 1
\end{array}\right).
\label{matrixop31}\end{equation}

Operator $\mathcal{L}_{KG}$ in
$(\ref{matrixop31})$, is obtained by considering the conserved
quantities $E$ and $F$ defined in $(\ref{E})$ and $(\ref{F})$ respectively. By defining
$G=E-cF$, one has
$$G'(h,ch')
=E'(h,ch')
-cF'(h,ch')=0,$$ that is,
$(h,ch')$ is a critical point of
$G$. In addition, we have
$
\displaystyle\mathcal{L}_{KG}=G''(h,ch')$ and
$(h',ch'')\in\ker(\mathcal{L}_{KG})$.

\indent Let us consider the quadratic form associated with matrix
operator $(\ref{matrixop31})$,
\begin{equation}\begin{array}{lll}
Q_{KG}(u,v)&=&\displaystyle\langle\mathcal{L}_{KG}(u,v),(u,v)\rangle_X\\\\
&=&\displaystyle\int_{0}^{L}\omega u'^2-u^2+(2k+1)h^{2k}u^2dx+
||c u'-v||_{L_{per}^2}^2\\\\
&=&Q(u)+||c u'-v||_{L_{per}^2}^2,
\end{array}\label{quadraticSG}\end{equation}
where for $\omega=1-c^2>0$, we have that
\begin{equation}
Q(u):=\int_{0}^{L}\omega u'^2-u^2+(2k+1)h^{2k}u^2dx, \label{quadraticDN1}
\end{equation}
represents the quadratic form of the operator $\mathcal{L}_{\omega,B}$ in $(\ref{operator})$.

We have the following results concerning the linearized operator $\mathcal{L}_{KG}$ defined in $(\ref{matrixop31})$.
\begin{proposition}   \label{teoeigenKG}
	Let $\omega>0$ be fixed and consider $k\geq1$ an integer. Let $h=h_{\omega,B}$ be the periodic wave solution determined in Section 2. The
	operator $\mathcal{L}_{KG}$ defined in $L_{per}^2([0,L])\times L_{per}^2([0,L])$, with domain $H_{per}^2([0,L])\times H_{per}^1([0,L])$ has zero as the second eigenvalue which is simple. Moreover, the remainder of the spectrum is a discrete set which is bounded away from zero.
\end{proposition}
\begin{proof}
	
	In fact, Remark $(\ref{obs1.1})$ enables us to
	say that the operator $\mathcal{L}_{\omega,B}$ in $(\ref{operator})$ has
	exactly one negative eigenvalue which is simple and zero is a simple
	eigenvalue with eigenfunction $h'$. Let $d$ be the
	unique negative eigenvalue of $\mathcal{L}$ with eigenfunction
	$\upsilon$. Since $Q(\upsilon)=d||\upsilon||_{L_{per}^2}^2<0$, we see from
	$(\ref{quadraticSG})$ that $$
	Q_{KG}(\upsilon,c\upsilon')=Q(\upsilon)+\displaystyle
	||c\upsilon'-c\upsilon'||_{L_{per}^2}^2=Q(\upsilon)<0.
	$$
	Moreover, the smallest eigenvalue $\sigma_1$ associated to $\mathcal{L}_{KG}$,
	 is negative. We establish that the next eigenvalue
	of $\mathcal{L}_{KG}$ is $\sigma_2:=0$ (which is simple) and also,
	the third eigenvalue $\sigma_3$, is strictly positive. To show
	these facts, we need to use the min-max characterization of eigenvalues. Indeed, in the energy space $X$, we have
	\begin{equation}
	\sigma_2=\max_{(f_1,f_2)\in X}\min_{(u,v)\in
		X\setminus\{0\}\atop{(u,f_1)_{H_{per}^1}+(v,f_2)_{L_{per}^2}=0}}
	\frac{Q_{KG}(u,v)}{||u||_{H_{per}^1}^2+||v||_{L_{per}^2}^2}.
	\label{minmax2}\end{equation} Then, by considering $f_1=\upsilon$
	and $f_2=0$ we get,
	\begin{equation}
	\sigma_2\geq\min_{(u,v)\in X\setminus\{0\}\atop{(u,\upsilon)_{H_{per}^1}=0}}
	\frac{Q_{KG}(u,v)}{||u||_{H_{per}^1}^2+||v||_{L_{per}^2}^2}\geq0 \label{min2}
	\end{equation}
	and therefore, $\sigma_2=0$. The proof that $\sigma_3>0$ is obtained
	from the same arguments used above when we take the two-dimensional
	subspace spanned by $(\upsilon,0)$ and $(h',0)$ since in this
	case $Q(u)\geq\overline{\sigma_3}||u||_{L_{per}^2}^2$, for
	$u\bot\upsilon$, $u\bot h'$, where $\overline{\sigma_3}$ is the
	third eigenvalue related to $\mathcal{L}$ which is obviously
	positive. Therefore, we conclude that
	$\mathcal{L}_{KG}$ has one negative eigenvalue which is simple and
	zero is a simple eigenvalue with eigenfunction
	$(h',ch'')$ as desired.

\end{proof}

\begin{corollary}   \label{teoeigenKGodd}
	In the same framework of Proposition $\ref{teoeigenKG}$, the linearized operator $\mathcal{L}_{KG}$ at $c=0$ defined in $L_{per,odd}^2([0,L])\times L_{per,odd}^2([0,L])$, with domain $H_{per,odd}^2([0,L])\times H_{per,odd}^1([0,L])$ has no negative eigenvalues and $\ker(\mathcal{L}_{KG})=\{(0,0)\}$. Moreover, the remainder of the spectrum is a discrete set which is bounded away from zero.
\end{corollary}
\begin{proof}
Indeed, the linearized operator $\mathcal{L}_{\omega,B}$ at $c=0$ in $(\ref{operator})$ does not have negative eigevalues restricted to the odd sector. In addition, since $h$ is odd and $h'$ is even, one has that $\ker(\mathcal{L}_{\omega,B})=\{0\}$. Therefore, we obtain the existence of $\sigma>0$ such that $Q(u)\geq \sigma||u||_{L_{per}^2}^2$ for all $u\in H_{per,odd}^1([0,L])$. The result then follows by a direct application min-max characterization of eigenvalues and the definition of $Q_{KG}$ given by $(\ref{quadraticSG})$.
\end{proof}

\section{Orbital Stability and Instability of Odd Periodic Waves.}

\indent Since the equation $(\ref{gKG})$ is invariant under translations, we define the orbit generated by $(h,ch')$ as
\begin{equation}\label{orbit}
\Omega_{(h,ch')}=\{(h(\cdot+r),ch'(\cdot+r);\ r\in\mathbb{R}\}.
\end{equation}
In $X$, we introduce the pseudo metric $d$ by
$$d((u_1,v_1),(u_2,v_2))=\inf\{||(u_1,v_1)-(u_2(\cdot+r),v_2(\cdot+r))||_{X},r\in\mathbb{R}\}.$$
It is to be observed that, by definition, the distance between $(u_1,v_1)$ and $(u_2,v_2)$ is measured by the distance between $(u_1,v_1)$ and the orbit generated by $(u_2,v_2)$.

\begin{definition}\label{stadef}
	Let $(h,ch')$ be a traveling wave solution for \eqref{gKG}. We say that $(h,ch')$ is orbitally stable in $X$ provided that, given $\ve>0$, there exists $\delta>0$ with the following property: if $(\phi_0,\phi_1)\in X$ satisfies $\|(\phi_0,\phi_1)-(h,ch')\|_{X}<\delta$ and $(\phi,\psi)$ is solution of \eqref{gKG} in some local interval $[0,T_0)$  with initial condition $(\phi_0,\phi_1)$, then the solution can be continued to a solution in $0\leq t<+\infty$ and satisfies
	$$
	d((\phi(t),\psi(t)),\Omega_{(h,ch')})<\ve, \qquad \mbox{for all}\,\, t\geq0.
	$$
	Otherwise, we say that $(h,ch')$ is orbitally unstable in $X$.
\end{definition}

\indent Definition above establishes that the pair evolution $(\phi,\psi)$ which solves equation $(\ref{gKG})$ must exist in the energy spaces $X$ and $X_{odd}$ for all values of time $t$ in a local interval of time $[0,T_0)$, where $T_0>0$ is the maximal time where the solution exist. This can be done using the standard Kato theory for quasilinear equations in \cite{kato}.

\subsection{Orbital Stability in the Odd Sector}
In this subsection, we are going to prove that the periodic wave $(h,ch')$ determined in Section 2 is stable in the energy space $X_{odd}$ for $c=0$. The main tool is to use the approach contained in \cite{grillakis1}. In fact, according with the Corollary $\ref{teoeigenKGodd}$ we see that 
\begin{equation}\label{quadraticQ}Q(u)=(\mathcal{L}u,u)_{L_{per}^2}\geq \sigma ||u||_{L_{per}^2}^2.\end{equation}
\indent Let $a,b>0$ be arbitrary, since $h$ is a bounded odd periodic wave, we have

\begin{equation}\label{abnorm}\begin{array}{lllll}
a||u_x||_{L_{per}^2}^2+b||u||_{L_{per}^2}^2&\leq a||u_x||_{L_{per}^2}^2+\displaystyle\frac{b}{\sigma}(\mathcal{L}u,u)_{L_{per}^2}\\\\
&\leq \displaystyle \left(a+\frac{b}{\sigma}\omega\right)(\mathcal{L}u,u)+\displaystyle\left(a+\frac{b}{\sigma}\omega+M_k\right)||u||_{L_{per}^2}^2,
\end{array}\end{equation}
where $M_k=(2k+1)\max_{x\in[0,L]} |h(x)|^{2k}$. Thus, 
\begin{equation}\label{abnorm1}\begin{array}{lllll}
\displaystyle a||u_x||_{L_{per}^2}^2+\left(b-a-\frac{b}{\sigma}\omega-M_k\right)||u||_{L_{per}^2}^2
&\leq \displaystyle \left(a+\frac{b}{\sigma}\omega\right)(\mathcal{L}u,u).
\end{array}\end{equation}
Choosing $a,b>0$ such that $b-a-\frac{b}{\sigma}\omega-M_k>0$, we see that
\begin{equation}\label{quadraticQ1}Q(u)=(\mathcal{L}u,u)_{L_{per}^2}\geq \gamma ||u||_{H_{per}^1}^2,\end{equation}
for some $\gamma=\gamma(\sigma,k,a,b,\omega)>0.$ Now, consider $c=0$. From the definition of $Q_{KG}$ in $(\ref{quadraticSG})$, we have
\begin{equation}\label{quadraticQKG}Q(u,v)=\langle\mathcal{L}_{KG}(u,v),(u,v)\rangle_X\geq \gamma ||u||_{H_{per}^1}^2+||v||_{L_{per}^2}^2\geq \widetilde{\gamma}||(u,v)||_{X}^2,\end{equation}
where $\widetilde{\gamma}=\min\{\gamma,1\}>0$. Inequality $(\ref{quadraticQKG})$ is the cornerstone to prove our orbital stability result in the restricted subspace $X_{odd}$ as follows:

\begin{theorem}\label{teostab1}The periodic wave $(h,0)$ is orbitally stable in $X_{odd}$ in the sense of Definition $\ref{stadef}$ (without considering the translation symmetry).
\end{theorem}

\begin{proof}
The proof follows by $(\ref{quadraticQKG})$ combined with a direct application of the Theorem 3.3 in \cite{grillakis1}. 
\end{proof}

\subsection{Instability for the $\phi^4$ Equation.} Let $L_0>0$ be fixed. In this subsection, we use an explicit form of the wave $h$ to prove the orbital instability result for the case $k=1$. In fact, using the ansatz 
\begin{equation}h(x)=a{\rm sn}\left(\frac{4K(\kappa)x}{L_0},\kappa\right)
\label{SNsol}\end{equation} into the equation $(\ref{ode})$, one has that $h$ is an explicit periodic solution with 
\begin{equation}\label{vaSN}
a=\frac{\sqrt{2}\kappa}{\sqrt{\kappa^2+1}}.
\end{equation}
The value of $\omega$ can be explicitly determined as
\begin{equation}\label{omegaSN}
\omega=\frac{L_0^2}{16K(\kappa)^2(1+\kappa^2)}.
\end{equation}
\indent Since $\omega=1-c^2$, we see that $c>0$ is given by
\begin{equation}\label{cSN}
c=\frac{\sqrt{16(1+\kappa^2)K(\kappa)^2-L_0^2}}{4K(\kappa)\sqrt{1+\kappa^2}}.
\end{equation}
Let $L_0>0$ be fixed such that $0<\omega<\frac{L_0^2}{4\pi^2}$. There exists a smooth curve $$\omega\in \left(0,\frac{L_0^2}{4\pi^2}\right)\mapsto h_{\omega}$$ of periodic waves which solves the equation $(\ref{ode})$. The existence of a smooth curve with fixed period is crucial in our analysis since the function $d$ in terms of $c$ as $d(c)=E(h,ch')-cF(h,ch')$ is well defined. Next, since $d'(c)=-F(h,ch')=-c\int_0^{L_0}(h'(x))^2dx$, we see that 

\begin{equation}\label{d2}
\begin{array}{lllll}d''(c)
&=&\displaystyle-\int_0^{L_0}(h'(x))^2dx+2c^2\frac{d}{d\omega}\int_0^{L_0}(h'(x))^2dx\\\\
&=&\displaystyle-\int_0^{L_0}(h'(x))^2dx+2(1-\omega)\frac{d}{d\omega}\int_0^{L_0}(h'(x))^2dx.
\end{array}\end{equation}
  On the other hand, using the explicit expression for $(\ref{SNsol})$, we obtain that
\begin{equation}\label{vadseg}
\begin{array}{lllll}d''(c)&=\displaystyle-\int_0^{L_0}(h'(x))^2dx+2(1-\omega)\frac{d}{d\kappa}\int_0^{L_0}(h'(x))^2dx\left(\frac{d\omega}{d\kappa}\right)^{-1}\\\\
&=\displaystyle-\frac{16a^2K(\kappa)}{L_0}\int_0^{K}{\rm cn}^2(x,\kappa){\rm dn}^2(x,\kappa)dx\\\\
&\displaystyle+\frac{32(1-\omega)}{L_0}\frac{d}{d\kappa}\left(a^2K(\kappa)\int_0^K {\rm cn}^2(x,\kappa){\rm dn}^2(x,\kappa)dx\right)\left(\frac{d\omega}{d\kappa}\right)^{-1}.
\end{array}\end{equation}
\indent The values of $a$ and $\omega$ given by $(\ref{vaSN})$ and $(\ref{omegaSN})$, respectively, can be used to get combined with formula 361.03 in \cite{byrd} that

$$
\begin{array}{llllllll}d''(c)&=\displaystyle-\frac{16a^2K(\kappa)}{3\kappa^2L_0}\left((1+\kappa^2)E(\kappa)-(1-\kappa^2)K(\kappa)\right)\\\\
&+\displaystyle\frac{32}{3L_0}(1-\omega)\frac{d}{d\kappa}\left(\frac{a^2K(\kappa)}{\kappa^2}\left((1+\kappa^2)E(\kappa)-(1-\kappa^2)K(\kappa)\right)\right)\left(\frac{d\omega}{d\kappa}\right)^{-1}\\\\
&=-\displaystyle\frac{32}{3(1+\kappa^2)L_0}\left((1+\kappa^2)E(\kappa)-(1-\kappa^2)K(\kappa)\right)\\\\
&+\displaystyle\frac{64}{3L_0}(1-\omega)\frac{d}{dk}\left(\frac{K(\kappa)}{(1+\kappa^2)}\left((1+\kappa^2)E(\kappa)-(1-\kappa^2)K(\kappa)\right)\right)\left(\frac{d}{d\kappa}\left(\frac{1}{K(\kappa)^2(1+\kappa^2)}\right)\right)^{-1}.\\\\
\end{array}$$
\indent Since $p(k):=(1+\kappa^2)E(\kappa)-(1-\kappa^2)K(\kappa)>0$ for all $\kappa\in(0,1)$ and $1-\omega>0$, we only need to study the behavior of the last two terms containing derivatives with respect to $\kappa$. In fact, let us denote
$$q(k):=\frac{d}{dk}\left(\frac{K(\kappa)}{(1+\kappa^2)}\left((1+\kappa^2)E(\kappa)-(1-\kappa^2)K(\kappa)\right)\right)\left(\frac{d}{d\kappa}\left(\frac{1}{K(\kappa)^2(1+\kappa^2)}\right)\right)^{-1}.$$
Clearly, $q$ is a negative function in terms of $k\in(0,1)$, and thus $d''(c)<0$.\\
\indent Analysis above gives the following orbital instability result.

\begin{theorem}\label{teostab2} Let $L_0>0$ be fixed. For 
	$\omega\in\left(0,\frac{L_0^2}{4\pi^2}\right)$, consider the odd periodic solution $h$ given by $(\ref{SNsol})$. The periodic wave $(h,ch')$ is orbitally unstable in $X$ in the sense of Definition $\ref{stadef}$.
\end{theorem}

\begin{proof}
	By Propositions $\ref{teoisoL}$ and $\ref{teoeigenKG}$, we guarantee the existence of only one negative for $\mathcal{L}_{KG}$ and it results to be simple. In addition, zero is a simple eigenvalue corresponding to the eigenfunction $(h',ch'')$. The result now follows by Theorem 4.7 in \cite{grillakis1}.
\end{proof}

\subsection{Instability for the $\phi^6$ Equation.} This subsection is devoted in presenting the orbital instability associated to the equation $(\ref{gKG})$ for the case $k=2$.\\
\indent First, let us assume $(\ref{ode})$ with a general $k\geq1$ integer. Let $\omega_0>0$ and $B_0\in (0,B_{\omega_0})$ be fixed and consider $h_{\omega_0}$, the corresponding solution for the equation $(\ref{ode})$ with fixed period $L_0>0$. Suppose the existence of a smooth curve $\omega\in \mathcal{I}\mapsto h_{\omega}\in B_r(h_{\omega_0})$ of periodic waves with fixed period solving $(\ref{ode})$. Here, $\mathcal{I}\subset(0,+\infty)$ is an open neighborhood around $\omega_0=1-c_0^2>0$ and $B_r(h_{\omega_0})\subset H_{per}^2([0,L_0])$ is an open neighborhood around $h_{\omega_0}$. Our intention is to obtain a more convenient expression for $d''(c)$ in $(\ref{d2})$ to simplify the management of complicated expressions involving elliptic functions which appears in the case $k=2$. Comparing with the results obtained in the previous subsection for the case $k=1$, we have a considerable simplification of the calculations since we use few information about the solution $h$ in comparison with the expression in $(\ref{d2})$ and $(\ref{vadseg})$.

\indent In fact, let us define $c\in (-1,1)$ such that $c^2=1-\omega$, where $\omega\in\mathcal{I}$ has been determined in the last paragraph. Let $d$ be the function defined as $d(c)=E(h,ch')-cF(h,ch')$, where we again consider $h=h_{\omega}$ to simplify the notation. Since $d'(c)=-F(h,ch')=-c\int_0^{L_0}(h'(x))^2dx$, we see similarly to $(\ref{d2})$ that 

\begin{equation}\label{d21}
d''(c)
=\displaystyle-\int_0^{L_0}(h'(x))^2dx+2(1-\omega)\frac{d}{d\omega}\int_0^{L_0}(h'(x))^2dx.
\end{equation}
\indent We give a convenient expression for $D=\frac{d}{d\omega}\int_0^{L_0}(h'(x))^2dx$. The reason for that is to avoid heavy calculations as determined for the case $k=1$, where we used the explicit form in $(\ref{SNsol})$ to calculate $D$ is terms of the complete elliptic integrals of the first and second kinds. In fact, deriving equation $(\ref{ode})$ with respect to $\omega\in \mathcal{I}$ and multiplying the resulting equation by $h$, we obtain after integration over $[0,L_0]$ that
\begin{equation}\label{eq1}
-\omega\int_0^{L_0}\eta''h dx+\int_0^{L_0}h'^2 dx-\int_0^{L_0}\eta hdx+(2k+1)\int_0^{L_0}h^{2k+1}\eta dx=0,
\end{equation}
where $\eta=\frac{d h}{d \omega}$ and we omit the variable $x$ in $(\ref{eq1})$ to simplify the notation. An integration by parts applied to the first integral in $(\ref{eq1})$ establishes using the ODE $(\ref{ode})$ that
\begin{equation}\label{eq2}
\int_0^{L_0}h'^2dx+\frac{k}{k+1}\frac{d}{d\omega}\int_0^{L_0}h^{2k+2}dx=0.
\end{equation}
The same equation $(\ref{eq1})$ also gives
\begin{equation}\label{eq3}
\frac{\omega}{2}\frac{d}{d\omega}\int_0^{L_0}h'^2 dx+\int_0^{L_0}h'^2 dx-\frac{1}{2}\frac{d}{d\omega}\int_0^{L_0}h^2dx+\frac{2k+1}{2k+2}\frac{d}{d\omega}\int_0^{L_0}h^{2k+2} dx=0.
\end{equation}
\indent On the other hand, multiplying equation $(\ref{ode})$ by $h'$, we obtain after integration over $[0,L_0]$
\begin{equation}\label{eq4}
-\frac{\omega}{2}\int_0^{L_0}h'^2dx-\frac{1}{2}\int_0^{L_0}h^2dx+\frac{1}{2k+2}\int_0^{L_0}h^{2k+2}dx+AL=0,
\end{equation}
 where $A$ is constant of integration depending smoothly on $\omega$. Now, deriving equation $(\ref{eq4})$ with respect to $\omega\in \mathcal{I}$ and combining the final result with $(\ref{eq3})$, we obtain from $(\ref{eq2})$ that
 
 \begin{equation}\label{eq5}
 \omega\frac{d}{d\omega}\int_0^{L_0}h'^2dx-\frac{1}{2}\int_0^{L_0}h'^2dx-L_0\frac{dA}{d\omega}=0
 \end{equation}
 By $(\ref{eq1})$ and $(\ref{eq5})$, we deduce a simple expression for $d''(c)$ as
 \begin{equation}\label{eq6}
 d''(c)=-\frac{1}{\omega}\int_0^{L_0}h'^2dx+\frac{2(1-\omega)L_0}{\omega}\frac{dA}{d\omega}.
 \end{equation}
 \indent It is possible to obtain a convenient expression for $\frac{dA}{d\omega}$. Indeed, multiplying equation $(\ref{ode})$ by $h'$, integrating over $[0,x)$, using the fact that $h$ is odd, and deriving the final result with respect to $\omega$, we have
 \begin{equation}\label{eq7}
 \frac{dA}{d\omega}=\omega h'(0)\eta'(0)+\frac{h'(0)^2}{2}. 
 \end{equation}
 
Gathering $(\ref{eq6})$ and $(\ref{eq7})$, we get the more convenient equality
\begin{equation}\label{eq8}
d''(c)=-\frac{1}{\omega}\int_0^{L_0}h'^2dx+2L_0(1-\omega)h'(0)\eta'(0)+\frac{1-\omega}{\omega}h'^2(0)L_0.
\end{equation}

\indent Let $L_0>0$  be fixed. The explicit solution for the equation $(\ref{ode})$ with $k=2$ is given by

\begin{equation}\label{SNSNsol}
h(x)=\frac{a{\rm sn}\left(\frac{4K(\kappa)x}{L_0},\kappa\right)}{\sqrt{1-b{\rm sn}^2\left(\frac{4K(\kappa)x}{L_0},\kappa\right)}},
\end{equation}
where $a$, $b$ and $\omega$ are complicated functions in terms of $\kp$ given by
\begin{equation}\label{eqaSN}
a=\sqrt [4]{1458}\frac{\sqrt [4]{ \left(  \left( -1-{\kp}^{6}+\frac{3}{2}\,{\kp}^{
				4}+\frac{3}{2}\,{\kp}^{2} \right) \sqrt {s(\kp)}+ \left( s(\kp) \right) ^{2} \right)  \left( s(k) \right) ^{3}}}{{s(k)}},
\end{equation}
\begin{equation}\label{bSNSN}
b=\frac{1}{3}\kp^2+\frac{1}{3}-\frac{1}{3}\sqrt{\kp^4-\kp^2+1},
\end{equation}
and
\begin{equation}\label{omegaSNSN}
\omega=\frac{L_0^2}{16K(\kappa)^2\sqrt{\kp^4-\kp^2+1}},
\end{equation}
where the notation $s(\kp)=\kp^4-\kp^2+1$ in $(\ref{eqaSN})$ was employed to simplify the notation. Let $L_0>0$ be fixed. Equality in $(\ref{omegaSNSN})$ enables us to deduce that $\omega>0$ must be considered over the interval $\left(0,\frac{L_0}{4\pi^2}\right)$.\\
\indent By Propositions $\ref{teoisoL}$ and $\ref{teoeigenKG}$, we see that $\mathcal{L}_{KG}$ has only one negative eigenvalue which is simple, zero is a simple eigenvalue associated to the eigenfunction $(h',ch'')$ and the remainder of the spectrum consists in a discrete set which is bounded away from zero. It remains to calculate $d''(c)$ using the simplified formula in $(\ref{eq8})$. In fact, we see from the chain rule that
\begin{equation}\label{eq9}
d''(c)=-\frac{1}{\omega}\int_0^{L_0}h'^2dx+2L_0(1-\omega)h'(0)\frac{d\eta'(0)}{d\kp}\left(\frac{d\omega}{d\kp}\right)^{-1}+\frac{1-\omega}{\omega}h'^2(0)L_0.
\end{equation}
Consider $\beta(\kp):=2L_0(1-\omega)h'(0)\frac{d\eta'(0)}{d\kp}\left(\frac{d\omega}{d\kp}\right)^{-1}+\frac{1-\omega}{\omega}h'^2(0)L_0$. We are going to prove that $\beta(\kp)<0$ for all $\kp\in (0,1)$. In fact, an exhaustive calculation gives us that

\begin{equation}\label{betakp}
\beta(\kp)=\frac{72\, \left( \kp+1 \right) {\kp}^{2}}{L_0^3} \left( -16\,  {\it K}
\left( \kp \right)^{2}\sqrt {{\kp}^{4}-{\kp}^{2}+1}+L_0 \right) \tau(\kp),
\end{equation}
where $\tau(\kappa)$ is complicated smooth function depending only on $\kappa\in(0,1)$. The Figure $\ref{pic1}$ shows that $\tau(\kp)>0$ for all $\kp\in (0,1)$. 

\begin{figure}[!htb]
		\includegraphics[scale=0.35]{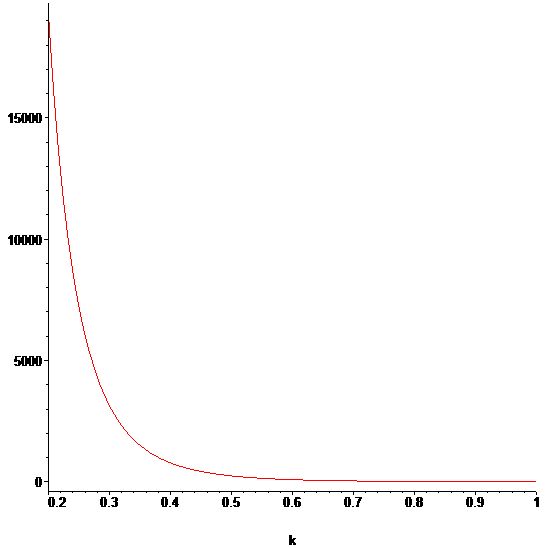}
		\caption{Graphic of $\tau$.}
\label{pic1}\end{figure}
\indent On the other hand, for a fixed $L_0>0$ such that $\omega\left(0,\frac{L_0}{4\pi^2}\right)$, one has $$ -16\,  {\it K}
\left( \kp \right)^{2}\sqrt {{\kp}^{4}-{\kp}^{2}+1}+L_0 <0,\ \ \ \ \forall\ \kp\in(0,1).$$ Therefore $d''(c)<0$ and we are in position to enunciate the following result.

\begin{theorem}\label{teostab3} Let $L_0>0$ be fixed. For 
	$\omega\in\left(0,\frac{L_0^2}{4\pi^2}\right)$, consider the odd periodic solution $h$ given by $(\ref{SNSNsol})$. The periodic wave $(h,ch')$ is orbitally unstable in $X$ in the sense of Definition $\ref{stadef}$.
\end{theorem}

 \section*{Acknowledgments}

 The first author dedicates this work to Professor Luciene Parron Gimenes Arantes. F. Natali is partially supported by Funda\c c\~ao Arauc\'aria/Brazil (grant 002/2017) and CNPq/Brazil (grant 304240/2018-4).

\end{document}